\documentclass{mystyle}
\usepackage{graphicx}
\usepackage{amsfonts, amssymb, amsmath, mathrsfs, latexsym}
\usepackage[all]{xy}
\usepackage[dvipsnames]{xcolor}
\usepackage{mathtools}
\usepackage{tikz}
\usetikzlibrary{positioning}
\usepackage{xcolor}
\usepackage{tikz-cd}
\usepackage[normalem]{ulem}

\ifpdf
	\usepackage[unicode,pdftex]{hyperref}
	\input glyphtounicode
\else
	\usepackage[unicode,dvipdfm]{hyperref}
\fi

\newtheorem{thm}{Theorem}[section]
\newtheorem{prop}[thm]{Proposition}
\newtheorem{lem}[thm]{Lemma}

\theoremstyle{definition}

\newtheorem{rmk}[thm]{Remark}

\newcommand{\dashdownarrow}{\mathrel{\rotatebox[origin=t]{-270}{\reflectbox{$\dashrightarrow$}}}}

\newcommand{\downinclusion}{\mathrel{\rotatebox[origin=t]{-90}{\reflectbox{$\subset$}}}}
\newcommand{\vequal}{\mathrel{\rotatebox[origin=t]{90}{\reflectbox{$=$}}}}
\newcommand{\downelement}{\mathrel{\rotatebox[origin=t]{-90}{\reflectbox{$\in$}}}}

\numberwithin{equation}{thm}
\newcommand{\vni}{\vskip 4pt \noindent}

\newcommand{\PP}{\ensuremath{\mathbb{P}}}

\newcommand{\h}{\ensuremath{\mathcal}}
\newcommand{\xdashrightarrow}[2][]{\ext@arrow 0359\rightarrowfill@@{#1}{#2}}
\begin{document}

\title[On  the Hilbert scheme of  smooth curves in $\mathbb{P}^4$ of degree $g+1$]
{On  the Hilbert scheme of  smooth curves in $\mathbb{P}^4$ of degree $d = g+1$  and genus $g$ with negative Brill-Noether number}

\thanks{This paper has started and prepared for publication when the first named author was enjoying the hospitality and the stimulating atmosphere of the Max-Planck-Institut f\"ur Mathematik (Bonn).
Also  supported  in part by National Research Foundation of South Korea (Grant \# 2019R1I1A1A01058457). The authors wish to thank the referee who pointed out non-trivial and rather silly mistakes in the submitted version. He also made a kind suggestion how to fix the error so that the main result of the paper still holds.}

 
\author[Changho Keem]{Changho Keem}
\address{
Department of Mathematics,
Seoul National University\\
Seoul 151-742,  
South Korea}

\email{ckeem1@gmail.com}

\thanks{}

\author[Yun-Hwan Kim]{Yun-Hwan Kim}
\address{Department of Mathematics,
Seoul National University\\
Seoul 151-742, 
South Korea}
\email{yunttang@snu.ac.kr}



\date{\today}
\begin{abstract}
We denote by $\mathcal{H}_{d,g,r}$ the Hilbert scheme of smooth curves, which is the union of components whose general point corresponds to a smooth irreducible and non-degenerate curve of degree $d$ and genus $g$ in $\PP^r$. 
In this article, we show that for low genus $g$ outside the Brill-Noether range, the Hilbert scheme $\mathcal{H}_{g+1,g,4}$ is non-empty whenever $g\ge 9$ and irreducible whose only component generically consists of linearly normal  curves unless $g=9$ or $g=12$. 
This complements the validity of the original assertion of Severi  regarding the irreducibility of $\mathcal{H}_{d,g,r}$ outside the Brill-Nother range  for $d=g+1$ and $r=4$.
\end{abstract}
\maketitle

\section{\quad An overview, preliminaries and basic set-up}

Given non-negative integers $d$, $g$ and $r\ge 3$, let $\mathcal{H}_{d,g,r}$ be the Hilbert scheme of smooth curves parametrizing smooth irreducible and non-degenerate curves of degree $d$ and genus $g$ in $\PP^r$.

It seems that the irreducibility of $\mathcal{H}_{d,g,r}$ was first announced by Severi in the papers \cite{Sev1} \& \cite{Sev}. Severi asserts with an incomplete proof that $\mathcal{H}_{d,g,r}$ is irreducible for $d\ge g+r$ and more generally,  $\mathcal{H}_{d,g,r}$ is irreducible in the Brill-Noether range $$\rho (d,g,r):=g-(r+1)(g-d+r)\ge 0.$$

\vskip 4pt
After Severi, the irreducibility of $\mathcal{H}_{d,g,r}$ has been studied by several authors. Ein proved irreducibility of $\mathcal{H}_{d,g,r}$ in the range $d\ge g+r$ for $r=3$ and $r=4$; cf.  \cite[Theorem 4]{E1} 
and \cite[Theorem 7]{E2}. 

\vskip 4pt
For families of curves in $\mathbb{P}^3$ of lower degree $d\le g+2$, the most updated result is that any non-empty $\mathcal{H}_{d,g,3}$ is irreducible for every $d\ge g$; cf. \cite[Theorem 1.5]{KK},  \cite[Proposition 2.1 and Proposition 3.2 ]{KKL}, \cite[Theorem 3.1]{I} and \cite{KKy1}. 
Note that even in the range $d\ge g$ and $r=3$, the Brill-Noether number  may well become negative, however there is no reducible $\h{H}_{d,g,3}$ in this range.  The reducible examples of $\mathcal{H}_{d,g,3}$ occur only when $d\le g-1$, e.g. $\h{H}_{8,9,3}$ or $\h{H}_{9,10,3}$ both of which have two components.  In every such known reducible example of $\h{H}_{d,g,3}$ when $d\le g-1$, the Brill-Noether numbers  is negative.

\vskip 4pt
For families of curves in $\PP^4$ of lower degree $d\le g+3$ beyond the range $d\ge g+4$, there has been some extensions of the result of Ein. Hristo Iliev proved the irreducibility of $\mathcal{H}_{d,g,4}$ for $d=g+3$, $g\ge 5$ and $d=g+2$, $g\ge 11$; cf. \cite{I}. Note that the genus restriction on the genus $g$ in these results is equivalent to the condition $\rho(g+3,g,4) > 0$ or $\rho(g+2, g,4)>0$. The non-negativity (or positivity) of the Brill-Noether number indeed assures the existence of a distinguished unique component of the Hilbert scheme dominating the moduli space $\mathcal{M}_g$ and sometimes this makes the problem rather easier to handle. However, outside the Brill-Noether range, i.e. in the range $\rho (d,g,r)<0$, the problem usually becomes more subtle and one needs somewhat thorough knowledge or {\it case by case} analysis of the classes of projective curves under consideration.

For a fixed value $d$ near to the genus $g$ and small $r$, the genus becomes rather small if the Brill-Noether number is negative. Indeed methods of the proof of the irreducibility of  $\mathcal{H}_{g+2,g,4}$ for low genus cases - which is an extension of the result of H. Iliev regarding the irreducibility of $\mathcal{H}_{g+2,g,4}$ - significantly differ from those in the Brill-Noether range; cf. \cite[Corollary 2.2]{KKy2}.

\vskip 4pt
On the other hand, there are several examples violating the original Severi's assertion due to many authors; cf. \cite[Proposition 9]{E2}, \cite{Keem}, or \cite{CKP}. Despite all these {\it denumerably many} counterexamples to the  Severi's assertion, it is worthwhile to stress that all the reducible counterexamples (in the Brill-Noether range) are those such that all the extra components consist of non linearly normal curves. In this regard, we denote by $\mathcal{H}^\mathcal{L}_{d,g,r}$ the union of those components of ${\mathcal{H}}_{d,g,r}$ whose general element is linearly normal. For a detailed explanation and sources in the literature regarding the
Hilbert scheme of linear normal curves $\mathcal{H}^\mathcal{L}_{d,g,r}$, the reader are advised to refer \cite{BFK} or \cite[Remark 1.1]{KK3} and references therein.

In this article, we show that $\mathcal{H}_{g+1,g,4}$ is non-empty for $g\ge 9$ and irreducible except for $g=9$ and  $g=12$ outside the the Brill-Noether range; i.e. $g\le 14$. 
\vskip 4pt

The following remarks would help the readers to get down to the main issues of the Hilbert scheme business of this kind.

\begin{rmk}\label{review}
(1) It has been shown in \cite{KK3} that any non-empty $\h{H}^\h{L}_{g+1,g,4}$ is irreducible for $g\neq 9$, inside or outside the Brill-Noether range. Note that a non-empty $\h{H}_{d,g,4}$ is known to be irreducible only in the range $d\ge g+2$ by those results due to Ein, H. Illiev and the authors which were already mentioned.
\vskip 4pt
\noindent
(2) However, at least to the knowledge of the authors,  the irreducibility of $\h{H}_{g+1,g,4}$ has not been fully settled yet. In \cite{KK3} the irreducibility of $\h{H}^\h{L}_{g+1,g,4}$ was shown by the irreducibility of the Severi variety and the fact that the family of complete linear systems on moving curves corresponding to $\h{H}^\h{L}_{g+1,g,4}$ has the expected dimension, which is not directly applicable  to other possible components consisting of non-linearly normal curves.

\vskip 4pt
\noindent
(3) In this paper the author would like to make an attempt for a settlement of the irreducibility of 
$\h{H}_{g+1,g,4}$ outside the Brill-Noether range, and hopefully the methods we employ in this article may shed light on  handling the cases inside the Brill-Noether range as well.

\end{rmk}

The organization of this paper is as follows. After we briefly recall several basic preliminaries in the remainder of this section, we start the next section with the two reducible examples of $\mathcal{H}_{g+1,g,4}$ for $g=9$ and $g=12$.

\vskip 4pt
 We then proceed to deal with the irreducibility of $\mathcal{H}^\mathcal{L}_{g+1,g,4}$ for some low genus $g$, e.g.  $g=10$ or $g=11$. 
\vskip 4pt
For notations and conventions, we usually follow those in \cite{ACGH} and \cite{ACGH2}; e.g. $\pi (d,r)$ is the maximal possible arithmetic genus of an irreducible and non-degenerate curve of degree $d$ in $\PP^r$. Throughout we work over the field of complex numbers.

\vskip 4pt
Before proceeding, we recall several related results which are rather well-known; cf. \cite{ACGH2}.
Let $\mathcal{M}_g$ be the moduli space of smooth curves of genus $g$. For any given isomorphism class $[C] \in \mathcal{M}_g$ corresponding to a smooth irreducible curve $C$, there exist a neighborhood $U\subset \mathcal{M}_g$ of the class $[C]$ and a smooth connected variety $\mathcal{M}$ which is a finite ramified covering $h:\mathcal{M} \to U$, as well as  varieties $\mathcal{C}$, $\mathcal{W}^r_d$ and $\mathcal{G}^r_d$ proper over $\mathcal{M}$ with the following properties:
\begin{enumerate}
\item[(1)] $\xi:\mathcal{C}\to\mathcal{M}$ is a universal curve, i.e. for every $p\in \mathcal{M}$, $\xi^{-1}(p)$ is a smooth curve of genus $g$ whose isomorphism class is $h(p)$,
\item[(2)] $\mathcal{W}^r_d$ parametrizes the pairs $(p,L)$ where $L$ is a line bundle of degree $d$ and $h^0(L) \ge r+1$ on $\xi^{-1}(p)$,
\item[(3)] $\mathcal{G}^r_d$ parametrizes the couples $(p, \mathcal{D})$, where $\mathcal{D}$ is possibly an incomplete linear series of degree $d$ and dimension $r$ on $\xi^{-1}(p)$ - which is usually denoted by $g^r_d$. 
\end{enumerate}

\vskip 4pt
Let $\widetilde{\mathcal{G}}$ ($\widetilde{\mathcal{G}}_\mathcal{L}$ resp.) be  the union of components of $\mathcal{G}^{r}_{d}$ whose general element $(p,\mathcal{D})$ of $\widetilde{\mathcal{G}}$ ($\widetilde{\mathcal{G}}_\mathcal{L}$ resp.) corresponds to a very ample (very ample and complete resp.) linear series $\mathcal{D}$ on the curve $C=\xi^{-1}(p)$. Note that an open subset of $\mathcal{H}_{d,g,r}$ consisting of points corresponding to smooth irreducible and non-degenerate curves is a $\mathbb{P}\textrm{GL}(r+1)$-bundle over an open subset of $\widetilde{\mathcal{G}}$. Hence the irreducibility of $\widetilde{\mathcal{G}}$ guarantees the irreducibility of $\mathcal{H}_{d,g,r}$. Likewise, the irreducibility of $\widetilde{\mathcal{G}}_\mathcal{L}$ ensures the irreducibility of 
 $\mathcal{H}_{d,g,r}^\mathcal{L}$.

\vskip 4pt
We also make a note of the following well-known facts regarding the schemes $\mathcal{G}^{r}_{d}$ and $\mathcal{W}^r_d$; cf. \cite[Proposition 2.7, 2.8]{AC2}, \cite[2.a]{H1}, \cite[Ch. 21, \S 3, 5, 6, 11, 12]{ACGH2} and \cite[Theorem 1]{EH}. Following classical terminology, a base-point-fee linear series $g^r_d$ which is not very ample ($r\ge 2$) on a smooth curve $C$ is called {\it birationally very ample} when the morphism 
$C \rightarrow \mathbb{P}^r$ induced by  the $g^r_d$ is generically one-to-one (or birational) onto its image;  cf. \cite[p. 570]{EH2}, 

\begin{prop}\label{facts}
For non-negative integers $d$, $g$ and $r$, let $\rho(d,g,r):=g-(r+1)(g-d+r)$ be the Brill-Noether number.
	\begin{enumerate}
	\item[\rm{(1)}] The dimension of any component of $\mathcal{G}^{r}_{d}$ is at least $3g-3+\rho(d,g,r)$ which is denoted by $\lambda(d,g,r)$. Moreover, if $\rho(d,g,r)\geq0$, there exists a unique component $\mathcal{G}_0$ of $\widetilde{\mathcal{G}}$ which dominates $\mathcal{M}$(or $\mathcal{M}_g$).

	\item[\rm{(2)}] Suppose $g>0$ and let $X$ be a component of $\mathcal{G}^{2}_{d}$ whose general element $(p,\mathcal{D})$ is such that $\mathcal{D}$ is a birationally very ample linear series on $\xi^{-1}(p)$. Then 
	\[\dim X=3g-3+\rho(d,g,2)=3d+g-9.\]
		\end{enumerate}
\end{prop}
\begin{rmk}\label{principal}

(1)
In the Brill-Noether range, the unique component $\mathcal{G}_0$ of $\widetilde{\mathcal{G}}$ (and the corresponding component $\mathcal{H}_0$ of $\mathcal{H}_{d,g,r}$ as well) which dominates $\mathcal{M}$ or $\mathcal{M}_g$ is called the  ``principal component".  

\vskip 4pt
\noindent
(2) In the range $d\le g+r$ inside the Brill-Noether range, the principal component $\mathcal{G}_0$ which has the expected dimension is one of the components of $\widetilde{\mathcal{G}}_\mathcal{L}$ (cf. \cite[2.1 page 70]{H1}),   and therefore $\widetilde{\mathcal{G}}_\mathcal{L}$ or 
 $\mathcal{H}_{d,g,r}^\mathcal{L}$  and hence  $\mathcal{H}_{d,g,r}$ is non-empty. 
 If one can show that $\mathcal{G}_0$ is the only component of 
 $\widetilde{\mathcal{G}}_\mathcal{L}$ ( $\widetilde{\mathcal{G}}$ resp.), the irreducibility of 
 $\mathcal{H}_{d,g,r}^\mathcal{L}$ ( $\mathcal{H}_{d,g,r} $ resp.) would follow immediately. 

\vskip 4pt
\noindent
(3) However outside the Brill-Noether range, such a distinguished component 
does not exist and this is one of the reasons why the problem becomes rather subtle.
\end{rmk}

We will utilize the  following upper bound of the dimension of an irreducible component of $\mathcal{W}^r_d$, which was proved  and used effectively in \cite{I}. 
A base-point-free linear series $g^r_d$ on $C$  is called {\it compounded of an involution (compounded for short)} if the morphism induced by the linear series gives rise to a non-trivial covering map $C\rightarrow C'$ of degree $k\ge 2$. 

\begin{prop}[\rm{\cite[Proposition 2.1]{I}}]\label{wrdbd}
Let $d,g$ and $r\ge 2$ be positive integers such that  $d\le g+r-2$ and let $\mathcal{W}$ be an irreducible component of $\mathcal{W}^{r}_{d}$. For a general element $(p,L)\in \mathcal{W}$, let $b$ be the degree of the base locus of the line bundle $L=|D|$ on $C=\xi^{-1}(p)$. Assume further that for a general $(p,L)\in \mathcal{W}$ the curve $C=\xi^{-1}(p)$ is not hyperelliptic. If the moving part of $L=|D|$ is
	\begin{itemize}
	\item[\rm{(a)}] very ample and $r\ge3$, then 
	$\dim \mathcal{W}\le 3d+g+1-5r-2b$;
	\item[\rm{(b)}] birationally very ample, then 
	$\dim \mathcal{W}\le 3d+g-1-4r-2b$;
	\item[\rm{(c)}] compounded, then 
	$\dim \mathcal{W}\le 2g-1+d-2r$.
	\end{itemize}
\end{prop}

\section{\quad Irreducibility of $\mathcal{H}_{g+1,g,4}$}

We first recall several relevant results from \cite[Theorem 2.1 and Theorem 2.2]{KK3}.

\begin{rmk}\label{linearnormal}
\begin{enumerate}
\item[\rm{(1)}] Every non-empty $\mathcal{H}_{g+1,g,4}^\mathcal{L}$ is irreducible unless $g=9$. 
\item[\rm{(2)}] $\mathcal{H}_{g+1,g,4}=\mathcal{H}_{g+1,g,4}^\mathcal{L}=\emptyset$ for $g\leq8$.		
		\item[\rm{(3)}] For $g=9$, $\mathcal{H}_{10,9,4}=\mathcal{H}_{10,9,4}^\mathcal{L}$ is reducible with two components of dimensions $42$ and $43$.
		\item[\rm{(4)}] For $g=10$,  $\mathcal{H}_{11,10,4}=\mathcal{H}_{11,10,4}^\mathcal{L}$ is irreducible of the expected dimension $46$.
		\item[\rm{(5)}] For $g=12$, $\mathcal{H}_{13,12,4}$ is reducible with two components of the same expected dimension $54$, whereas 
		$\mathcal{H}_{13,12,4}^\mathcal{L}$ is irreducible.
\item[\rm{(5)}] The non-emptiness $\mathcal{H}^\h{L}_{g+1,g,4}$ (hence $\mathcal{H}_{g+1,g,4}$ as well) and the existence of the extra components in the reducible cases above is rather clear from the proof in \cite[Proposition 2.2]{KK3}. 
	\end{enumerate}

\end{rmk}

We make a note of the following  lemma which will avoid unnecessary repetitions in the course of the proof of our result.

\begin{lem}\label{double}
The residual series $\mathcal{E}=|K_C-\h{D}|$ of a  possibly incomplete very ample (and special) linear series $\h{D}$ cannot induce a double covering $C\stackrel{\eta}{\rightarrow} E$ onto a curve of genus $h\ge 0$ so that the base-point-free part of $\h{E}$ is a pull back of a non-special linear series on the target curve $E$ via $\eta$.
	
\end{lem}
\begin{proof} Note that the complete $|\mathcal{D}|=g^r_d$ is very ample if $\h{D}$ is. Let $C\stackrel{\eta}{\rightarrow}E\subset\PP^s$ be the double covering induced by the base-point-free part of $\h{E}=|K_C-\h{D}|$ and let $h$ be the genus of the curve $E$. We may assume $s:=\dim\h{E}\ge 2$ and $h\ge 1$ since for $s=1$ or $h=0$, $C$ is hyperelliptic and a hyperelliptic curve does not have a special very ample linear series. 

\vni Let $\Delta$ be the base locus of $\h{E}$ and  $|\h{E}-\Delta|=\eta^*(\h{F})$ for some complete, non-special and base-point-free $\h{F}=g^s_f$ on $E$ where $s=g-d+r-1$. Since $\h{F}$ is non-special, $s=f-h$. 
Choose any $t\in E$,  set $u+v:=\eta^*(t)$ and we have
\begin{eqnarray*} s+1&\le&\dim\eta^*(|g^s_f+t|)\le\dim|\eta^*(\mathcal{F}+t)|=\dim|\eta^*(\mathcal{F})+u+v|\\&=& \dim|\mathcal{E}-\Delta+u+v| = \dim|K_C-\mathcal{D}-\Delta +u+v|\\
&\le&
\dim|K_C-(\mathcal{D}-u-v)|.
\end{eqnarray*}
By Riemann-Roch, it follows that $$\dim |\mathcal{D}-u-v|=d-2-g+h^0(C, K_C-(\h{D}-u-v)))\ge r-1$$ and hence  $|\mathcal{D}|$ is not very ample, a contradiction. 
\end{proof}

Since the irreducibility (or the reducibility) of $\h{H}_{g+1,g,4}$ were treated fully in 
\cite{KK3} for $g=9, 10,12$ as in Remark \ref{linearnormal}, we focus on the three remaining cases $g=11$, $g=13$ and $g=14$ which are also outside the Brill-Noether range.

\vni
We will use the following dimension estimate of a component $\h{W}\subset\h{W}^r_{g+1}$ corresponding to a possible extra component of $\h{H}_{g+1,g,4}$ other than the one corresponding to the 
irreducible $\h{H}^\h{L}_{g+1,g,4}$.
\begin{lem} \label{easy}Assuming that $\h{H}_{g+1,g,4}$ is not irreducible,  we let $\mathcal{H}$ be an extra component of $\mathcal{H}_{g+1,g,4}$ other than $\mathcal{H}^\mathcal{L}_{g+1,g,4}$. Let $\mathcal{G}\subset\mathcal{G}^r_d$ be the component corresponding to $\mathcal{H}$, i.e. $\mathcal{H}$ is a $\mathbb{P}GL(5)$-bundle over $\mathcal{G}$ and $\h{G}\neq\h{G}_\h{L}$.
Let $\mathcal{W}\subset \mathcal{W}^{r}_{g+1}$ be the component containing the image of the natural rational map 
${\mathcal{G}}\overset{\iota}{\dashrightarrow} \mathcal{W}^{r}_{g+1}$ with $\iota (\mathcal{D})=|\mathcal{D}|$.
We also let $\mathcal{W}^\vee\subset \mathcal{W}^{r-2}_{g-3}$ be the locus consisting  of the residual  series  of elements in $\mathcal{W}$, i.e. $$\mathcal{W}^\vee =\{(p, \omega_C\otimes L^{-1}): (p, L)\in\mathcal{W}\}.$$ Let $b$ be the degree of the base locus of a general element of $\h{W}^\vee$.
Then $r\ge 5$ and we have the following dimension estimate of the locus $\h{W}$. 

\begin{enumerate}
\item[\rm{(1)}] If the base-point-free  part of a general element of $\h{W}^\vee$ is very ample,  then $b=0$ and
$$\dim\h{W}=\dim\h{W}^\vee=4g-5r+2.$$

\item[\rm{(2)}] If the base-point-free  part of a general element of $\h{W}^\vee$ is birationally very ample, then 
$$4g-5r+2\le \dim\mathcal{W}=\dim\mathcal{W}^{\vee}\le4g-4r-2-2b.$$
In particular, if $r=5$ we have $b=0$.
\item[\rm{(3)}] If the base-point-free  part of a general element of $\h{W}^\vee$ is compounded, 
$$4g-5r+2\le\dim\h{W}=\dim\h{W}^\vee\le3g-2r.$$
In particular we have $g\le 3r-2$.
\end{enumerate}
\end{lem}
\begin{proof} By Remark \ref{linearnormal}(1), a general element $(p,\mathcal{D})\in\mathcal{G}$ is such that $\mathcal{D}$ is an incomplete $g^4_{g+1}$ on $C=\xi^{-1}(p)$ hence $r:=\dim |\mathcal{D}|\ge 5$.

\vni
(1) By Proposition \ref{facts}(1) and Proposition \ref{wrdbd}(a), we have 
\begin{eqnarray*}
\lambda (g+1,g,4)&=&4g-18\le\dim\mathcal{G}\le\dim\mathbb{G}(4,r)+\mathcal{W}\\&=&\dim\mathbb{G}(4,r)+\dim \mathcal{W}^{\vee}\\
&\le& 5(r-4)+3(g-3)+g+1-5(r-2)-2b\\
&=&4g-18-2b 
\end{eqnarray*}
and hence $b=0$ and $$\dim\h{W}=\dim\h{W}^\vee=4g-18-5(r-4)=4g-5r+2.$$

\vni
(2)  By Proposition \ref{facts}(1) and Proposition \ref{wrdbd}(b), we have 
 \begin{eqnarray*}
4g-18&\le&\dim\mathcal{G}\le\dim\mathbb{G}(4,r)+\mathcal{W}=\dim\mathbb{G}(4,r)+\dim \mathcal{W}^{\vee}\\
&\le& 5(r-4)+3(g-3)+g-1-4 (r-2)-2b\\
&=&4g-22+r-2b.
\end{eqnarray*}
hence 
\begin{equation*}
4g-5r+2\le \dim\mathcal{W}=\dim\mathcal{W}^{\vee}\le4g-4r-2-2b.
\end{equation*}

\vni
(3) By Proposition \ref{wrdbd}(c), we have 
\begin{eqnarray*}
4g-18&\le&\dim\mathcal{G}\le\dim\mathbb{G}(4,r)+\mathcal{W}=\dim\mathbb{G}(4,r)+\dim \mathcal{W}^{\vee}\\
&\le&5(r-4)+ 2g-1+(g-3)-2(r-2)\\
&=&3g+3r-20.
\end{eqnarray*}
implying  
$$4g-5r+2\le\dim\h{W}=\dim\h{W}^\vee\le3g-2r.$$
\end{proof}

\begin{rmk} The following lemma is rather technical but is useful in transferring the  original problem into lower degree cases. Indeed it is a slight variation of a claim which appeared in the course of the proof of a key lemma  in \cite[Lemma 2.3]{KKy1}. For the convenience of the readers, a similar proof is provided.
\end{rmk}

\begin{lem}\label{diagram} Assuming that $\h{H}_{g+1,g,4}$ is not irreducible,  we let $\mathcal{H}$ be an extra component of $\mathcal{H}_{g+1,g,4}$ other than $\mathcal{H}^\mathcal{L}_{g+1,g,4}$. Let $\mathcal{G}\subset\mathcal{G}^r_d$ be the component corresponding to $\mathcal{H}$. 
Let $\mathcal{W}\subset \mathcal{W}^{r}_{g+1}$ and $\mathcal{W}^\vee\subset \mathcal{W}^{r-2}_{g-3}$ be as in Lemma \ref{easy}. Assume further that a general element 
of $\h{W}^\vee$ is base-point-free and birationally very ample.
Then there is a locus $\h{Z}\subset \h{W}^{r-3}_{g-5}$ such that 
$$\dim\h{Z}=\dim\h{W}=\dim{W}^\vee.$$

\end{lem}
\begin{proof}
We consider the following obvious diagram: 

\[
\begin{array}{ccc}
\mathcal{W}^{r-3}_{g-5}\underset{\mathcal{M}}{\times}\mathcal{W}_2 &\stackrel{q}\dashrightarrow
\mathcal{W}^{r-3}_{g-3}\\
\\
\dashdownarrow\vcenter{\rlap{$\scriptstyle{{\pi}}\,$}}
\\ 
\\
\mathcal{W}^{r-3}_{g-5}
\end{array}
\]
where $q(\mathcal{E}',\mathcal{O}_C(R+S))=\mathcal{E}'\otimes\mathcal{O}_C(R+S)$ and $\pi(\mathcal{E}',\mathcal{O}_C(R+S))=\mathcal{E}'$.  Since a general element $(p,\mathcal{E})\in\mathcal{W}^\vee\subset\mathcal{W}^{r-2}_{g-3}\subset\mathcal{W}^{r-3}_{g-3}$ is birationally very ample and base-point-free, $q^{-1}(\mathcal{E})\neq\emptyset$ for a general  $(p,\mathcal{E})\in\mathcal{W}^\vee$. Let $\Sigma$ be a component of $q^{-1}(\mathcal{W}^\vee)$ such that $q(\Sigma )=\mathcal{W}^\vee$. Since a general $(p,\mathcal{E})\in\mathcal{W}^\vee$ is assumed to be birationally very ample, we see that $\dim q^{-1}(\mathcal{E})=0$ and hence \[\dim\Sigma = \dim\mathcal{W}^\vee.\] We set 
$\mathcal{Z}:=\pi({\Sigma})\subset\mathcal{W}^{r-3}_{g-5}$  and consider the following induced diagram:
\[
\begin{array}{ccc}
\mathcal{W}^{r-3}_{g-5}\underset{\mathcal{M}}{\times}\mathcal{W}_2\supset\Sigma & \stackrel{q}\dashrightarrow &
\mathcal{W}^\vee\subset\mathcal{W}^{r-2}_{g-3}\subset\mathcal{W}^{r-3}_{g-3}\\
\\
\quad\quad\quad\quad\quad\quad\dashdownarrow\vcenter{\rlap{$\scriptstyle{{\pi}}\,$}} & 
\\ 
\\
\quad\quad\mathcal{W}^{r-3}_{g-5}\supset\mathcal{Z} &
\end{array}
\]
We now argue that $\dim \pi^{-1}(\mathcal{E}')=0$ for a general $(p,\mathcal{E}')\in\mathcal{Z}$ as follows.
\noindent
\noindent
We choose $(p,\mathcal{E})\in\mathcal{W}^\vee$ and fix $(p, \mathcal{E}')\in\mathcal{Z}$  such that $(\mathcal{E}', \mathcal{O}_C(R+S))\in q^{-1}(\mathcal{E})$ for some  $R, S\in C=\xi^{-1}(p)$, i.e. $\mathcal{E}\cong\mathcal{E}'\otimes\mathcal{O}_C(R+S)$. Recall that by our initial setting, $\omega_C\otimes\mathcal{E}^{-1}=|\mathcal{D}|\in\mathcal{W}\subset\mathcal{W}^r_{g+1}$ is a very ample line bundle for a general $\mathcal{E}\in\mathcal{W}^\vee\subset\mathcal{W}^{r-2}_{g-3}$. 
We also note that the very ample, base-point-free and complete linear system $|\mathcal{D}|=\omega_C\otimes \mathcal{E}^{-1}=\omega_C\otimes\mathcal{E}'^{-1}\otimes\mathcal{O}_C(-R-S)$ is a subsystem of 
$\omega_C\otimes\mathcal{E}'^{-1}$.  Hence $\omega_C\otimes\mathcal{E}'^{-1}$ is birationally very ample; otherwise the isomorphism induced by the very ample $\mathcal{D}$ on $C=\xi^{-1}(p)$ onto its image factors non-trivially through the morphism induced by $\omega_C\otimes\mathcal{E}'^{-1}$, which is an absurdity. Therefore by noting that $\omega_C\otimes\mathcal{E}'^{-1}=g^{r+1}_{g+3}$, 
there are only finitely many choices 
of  $\mathcal{O}_C(-\widetilde{R}-\widetilde{S})$'s such that 
$$\omega_C\otimes\mathcal{E}'^{-1}\otimes\mathcal{O}_C(-\widetilde{R}-\widetilde{S})=g^r_{g+1}\in \mathcal{W}\subset\mathcal{W}^r_{g+1},$$ i.e. $(\mathcal{E}', \mathcal{O}_C(\widetilde{R}+\widetilde{S}))\in\Sigma$ or equivalently $\mathcal{E}'\otimes\mathcal{O}_C(\widetilde{R}+\widetilde{S})\in\mathcal{W}^\vee$, which implies $\dim \pi^{-1}(\mathcal{E}')=0$. By semi-continuity, we have 
 $\dim \pi^{-1}(\mathcal{E}')=0$ for a general $(p,\mathcal{E}')\in\mathcal{Z}$  and hence
 \[
 \dim\mathcal{Z}=\dim\Sigma.
 \]
\end{proof}

The following theorem will be used in the proof of Theorem \ref{main} for the case $g=14$, which asserts that  $\h{H}_{11.14, 3}$ is reducible with exactly two components and each of them has the minimal possible dimension.

\begin{thm} \label{residual}The Hilbert scheme $\h{H}_{11.14, 3}$ is reducible with two components and both of them have the same expected dimension $44$.
\end{thm}
\begin{proof} Let $C$ be a smooth curve of genus $g=14$ and $d=11$ in $\PP^3$. By Riemann-Roch on $C$, the dimension of $H^0(\mathbb{P}^3,\mathcal{I}_C(4))$ is  at least 
$$h^0(\PP^3, \h{O}(4))-h^0(C,\h{O}(4))=35-(44-14+1)=4,$$ 
and hence $C$ lies on  quartic surfaces, which may well be reducible.

\vni
(1) Let $C$ lie on a (unique) quadric $Q$. By solving $d=11=a+b$ and $g=14=(a-1)(b-1)$, we see that $C$ is a curve of type $(a,b)=(3,8)$ on a smooth quadric. Let $\mathcal{I}_2$ be the family of such curves arising in this way, which is clearly irreducible. By counting the number of parameters of the family of  such pairs $(C,Q)$ we readily have  $$\dim\mathcal{I}_2=h^0(\mathbb{P}^3,\mathcal{O}(2))-1+h^0(\PP^1\times\PP^1, \mathcal{O}(a,b))-1
=44=4\cdot d.$$ It is not clear at this point if $\mathcal{I}_2$ is open and dense in a component of the Hilbert scheme $\h{H}_{11.14, 3}$, which will be answered in the affirmative at the end of the proof. 

\vni
(2) Next we assume that $C$ lie on a (unique) smooth cubic surface $S$, which is isomorphic to $\mathbb{P}^2$ blown up at $6$ points and embedded by anti-canonical linear system $|-K_S|$ in $\mathbb{P}^3$.  We denote by $\mathcal{O}_S(a;b_1,\cdots ,b_6)$ the invertible sheaf on $S$ associated to the Cartier divisor 
$al-\sum^6_{i=1}b_ie_i$ on $S$, where $l$ is the pull-back of a line in $\mathbb{P}^2$ and $e_i ~(1\le i\le 6)$ are the six exceptional divisors. Setting $C\sim al-\sum^6_{i=1}b_ie_i$, we have $$\deg C =3a-\sum b_i=11, C^2=a^2-\sum b_i^2=2g-2-K_S\cdot C=37.$$
By Schwartz's inequality,  one has $$(\sum b_i)^2\le 6(\sum b_i^2)$$ and substituting 
$\sum b_i=3a-11$, $\sum b_i^2=a^2-37$ we obtain $$3a^2-66a+343\le 0,$$ implying 
$9\le a\le 13$. After elementary but rather tedious numerical calculation, we arrive at the following possibilities for 
the $7$-tuple $(a; b_1, \cdots , b_7)$;

(i) $(9; 3,3,3,3,2,2)$ (ii) $(10; 4,4,3,3,3,2)$ (iii) $(11; 5,4,4,3,3,3)$ 

(iv) $(12; 5,5,4,4,4,3)$
(v) $(13; 5,5,5,5,4,4),$

\vni and by a simple numerical check, in all five cases the curve $C$ is linearly equivalent on $S$ to $D+3H$ where $D$ is a disjoint union of two of the $27$ lines on the cubic.

\vni
For $L=\mathcal{O}_S(a;b_1,\cdots ,b_6)$, by Riemann-Roch on $S$ we have
\begin{eqnarray*}
h^0(S,L)&=&h^1(S,L)-h^2(S,L)+\chi(S)+\frac{1}{2}(L^2-L\cdot\omega_S)\\
&=&h^1(S,L)-h^2(S,L)+1+\frac{1}{2}(C^2+\deg C)\\
&=&h^1(S,L)-h^2(S,L)+1+\frac{1}{2}(37+11)
\end{eqnarray*}

\vni
By Serre duality  we have, $$h^1(S,L)=h^1(S, \omega_S\otimes L^{-1})=h^1(S,\mathcal{O}_S(-(a+3)l+\sum(b_i+1)e_i).$$
Since $E:=(a+3)l-\sum(b_i+1)e_i$ is (very) ample we have $$h^1(S,L)=h^1(S,\mathcal{O}_S(-(a+3)l+\sum(b_i+1)e_i))=0$$ by Kodaira's vanishing theorem. 
We further note the divisor $-E$ is not linearly equivalent to an effective divisor; if it were, one would have $-E\cdot l=-(a+3)\ge 0$ whereas $l^2=1\ge 0$, a contradiction. Hence it follows that  $$h^2(S,L)=h^0(S,\omega_S\otimes L^{-1})=h^0(S,\mathcal{O}_S(-E)=0$$ and we obtain $$h^0(S,L)=25.$$ Therefore the sublocus $\mathcal{I}_3$ of $\mathcal{H}_{11,14,3}$ consisting of  curves lying on a smooth cubic has dimension strictly less than the minimal possible dimension of a component of the Hilbert scheme $\mathcal{H}_{11,14,3}$; 
$$\dim\h{I}_3=\dim H^0(\mathbb{P}^3, \mathcal{O}(3))-1+\dim |\mathcal{O}_S(a;b_1,\cdots ,b_6)|=43<4\cdot 11,$$ hence $\mathcal{I}_3$ is not dense in a component of  $\mathcal{H}_{11,14,3}$. Indeed $\mathcal{I}_3$ is in the boundary of the component $\overline{\mathcal{I}_4}$ (the closure of $\mathcal{I}_4$), which we are going to describe later.  On the other hand, 
we see easily that the locus $\mathcal{I}_3$ is not in the closure  $\overline{\mathcal{I}_2}$ by semicontinuity; note that a general member of the irreducible locus $\overline{\mathcal{I}_2}$ is trigonal whereas a general member $C$ in a component of the sublocus $\mathcal{I}_3$ has a base-point-free pencil of degree $6$ cut out on $C$ by the linear system $|2l-\sum^4_{i=1}b_ie_i|$. However such $C$ has gonality $k$ strictly greater than $3$ by Casteluovo-Severi inequality; if 
$k\le 3$, then $g=14\le (k-1)(6-1)\le 10$, an absurdity.

\vni
(3) It is possible that there might exist  components of $\mathcal{H}_{11,14,3}$ whose general element lies only on a singular cubic surface. However, one may argue that no such component exists as follows. 
Note that every singular cubic surface $S\subset \PP^3$ is one of the following three types.
\vskip 4pt
(i) $S$ is a normal cubic surface with some double points only.

(ii) $S$ is a normal cubic cone.

(iii) $S$ is not normal, which may possibly be a cone. 

\vni
For the case (i), let $S$ be a normal cubic surface which is not a cone. By a work due to John Brevik \cite[Theorem 5.24]{Brevik}, every curve on $S$ is a specialization of curves on a smooth cubic surface. Therefore we are done for the case (i). 

\vni
For the case (ii), we let $C$ be a smooth curve of degree $d$ and genus $g$ on a normal cubic cone $S$. Recall that 

(a) $g=1+d(d-3)/6-2/3$ if $C$ passes through the vertex of $S$ 

(b) $g=1+d(d-3)/6$, otherwise

\noindent
which can be found in \cite[Proposition 2.12]{Gruson} as an application of C. Segre formula.
However $(d,g)=(11,14)$ satisfies neither of the above. 

\vni
(iii) Let $C$ be a smooth curve of degree $d$ and genus $g$ on a non-normal cubic surface $S$. Recall that if $S$ is a cone, then $S$ is a cone over a singular plane cubic, in which case $S$ is a projection of a cone $S'$ over a twisted cubic in a hyperplane in  $\PP^4$ from a point not on $S'$. Furthermore,  the minimal desingularization $\tilde S$ of $S'$ is isomorphic to the ruled surface 
$$\mathbb{F}_3=\PP (\mathcal{O}_{\PP^1}\otimes \mathcal{O}_{\PP^1}(3)),$$ 
which is the blow-up of the cone $S'$ at the vertex.
If $S$ is not a cone, then $S$ is a projection of a rational normal scroll $$S''\cong\tilde S\cong\mathbb{F}_1=\PP (\mathcal{O}_{\PP^1}(1)\otimes \mathcal{O}_{\PP^1}(2))\subset \PP^4$$  from a point not on $S''$. 
In both cases, we have $\text{Pic}~ \tilde S=\mathbb{Z}h\otimes \mathbb{Z}f \cong\mathbb{Z}^{\otimes 2}$, where $f$ is the class of a fiber of $\tilde S \rightarrow \PP^1$ and $h=\pi^*(\mathcal{O}_S(1))$ with $\tilde S\stackrel{\pi}{ \rightarrow} S$. Note that $h^2=3, f^2=0$, $h\cdot f=1$ and $K_{\tilde S}\equiv -2h+f$. Denoting by $\tilde C\subset\tilde S$ the strict transformation of the curve $C\subset S$, we let $k:=(\tilde C \cdot f)_{\tilde S}$ be the intersection number of $\tilde C$ and $f$ on $\tilde S$. We have $\tilde C \equiv kh+(d-3k)f=kh+(11-3k)f$.
By adjuntion formula, it follows that
$$g=14=\frac{(2\cdot 11-3k-2)(k-1)}{2},$$ which does have an integer solution and we are done with the case (iii).

\vni
(3) Finally we assume that $C$ does not lie on a quadric or a cubic. Since every smooth curve $C$ of genus $g=14$ of degree $d=11$ in $\PP^3$ lies on at least $$h^0(\PP^3, \h{O}(4))-h^0(C,\h{O}(4))=35-(44-14+1)=4$$ independent quartics, we see in this case that $C$ is 
is residual to a curve $D\subset\PP^3$ of degree $e=5$ and genus $h=2$ in the complete intersection of two (irreducible) quartics
by the well-known formula relating degrees and genus of directly linked curves in $\PP^3$;
\begin{equation*}
2(g-h)=(s+t-4)(d-e).
\end{equation*}

\vni
Consider the locus
$$\Sigma\subset\mathbb{G}(1,\PP(H^0(\PP^3,\h{O}(4))))=\mathbb{G}(1,34)$$
of pencils of quartic surfaces whose base locus consists of a curve $C$ of degree $d=11$ and genus $g=14$ and a quintic $D$ of genus $h=2$ where $C$ and $D$ are directly linked via a complete intersection of quartics, together with the two obvious maps 
\[
\begin{array}{ccc}
\hskip -48pt\mathbb{G}(1,34)\supset\Sigma&\stackrel{\pi_C}\dashrightarrow
\quad\mathcal{I}_4\subset \hskip 3pt\h{H}_{11,14,3}\\
\\
\dashdownarrow\vcenter{\rlap{$\scriptstyle{{\pi_D}}\,$}}
\\ 
\\
\h{H}_{5,2,3},
\end{array}
\]
where $\mathcal{I}_4$ is the image of $\Sigma$ under $\pi_C$.
A quintic  $D\subset\PP^3$ of genus $h=2$ lies on at least $$h^0(\PP^3,\h{O}(4))-h^0(D,\h{O}(4))=35-19=16$$ independent quartics. 
Note that $C\in\mathcal{I}_4\subset\hskip 4pt\h{H}_{11.14.3}$ is directly linked to $D\in\h{H}_{5,2,3}$ which in turn is directly linked to a line $L$ via complete intersection of a quadric and a cubic. From the basic  relation
$$\dim H^1(\PP^3,\h{I}_C(m))=\dim H^1(\PP^3,\h{I}_D(s+t-4-m))$$
where $C$ and $D$ are directly linked via complete intersection of surfaces of degrees $s$ and $t$, we have 
\begin{equation}\label{h1}h^1(\PP^3,\h{I}_C(4))=h^1(\PP^3,\h{I}_D(0))=h^1(\PP^3,\h{I}_L(2+3-4-0))=0
\end{equation}
and 
$$h^1(\PP^3,\h{I}_D(4))=h^1(\PP^3,\h{I}_L(2+3-4-4))=0.$$
\vni
 Therefore $\pi_D$ is generically surjective with fibers open subsets of $\mathbb{G}(1,15)$.  Since $\dim\h{H}_{5.2.3}$ is known to be irreducible (cf.\cite{E1} or \cite[page 51]{H1}), it follows that  $\Sigma$ is irreducible and $$\dim\Sigma=\dim\mathbb{G}(1,15)+\dim\h{H}_{5.2.3}=28+4\cdot 5=48.$$ On the other hand, since every $C\in\mathcal{I}_4$ lies on exactly $4$ independent quartics by (\ref{h1}), $\pi_C$ is generically surjective with fibers open in $\mathbb{G}(1,3)$. Finally it follows that the locus $\mathcal{I}_4$ is irreducible of dimension
$$\dim\Sigma-\dim\mathbb{G}(1,3)=44=4\cdot11.$$
 
\vni
Since the irreducible loci $\mathcal{I}_2$ and $\mathcal{I}_4$ have the same dimension, one may deduce that general element in $\mathcal{I}_2$ is not  a specialization of a curve in $\mathcal{I}_4$ (and vice versa). Furthermore,  together with the irreducibility of the loci $\mathcal{I}_2$ and $\mathcal{I}_4$, this implies that both are dense in the two components $\overline{\mathcal{I}_2}$
and $\overline{\mathcal{I}_4}$ of $\h{H}_{11,14,3}$. We also remark that the locus $\mathcal{I}_3$\footnote{From the referee, the authors were informed of an alternative way to see that  a curve $C$ in $\mathcal{I}_3$ is a limit of curves not lying on a cubic surface using results by Martin-Deschamps and Perrin \cite[Proposition 2.3, Lemma 2.4 and Proposition 4.1]{MP}. } sits in the boundary of the second component $\overline{\mathcal{I}_4}$ as was explained earlier.
\end{proof}

\begin{thm}\label{main} For $g=11, 13$ and $14$
\begin{enumerate}
\item[\rm{(1)}] 
$\mathcal{H}_{g+1,g,4}$ is irreducible and $\mathcal{H}_{g+1,g,4}=\mathcal{H}^\h{L}_{g+1,g,4}$,
\item[\rm{(2)}]
$\dim\mathcal{H}_{g+1,g,4}=\lambda(g+1,g,4)+\dim\PP GL(5)$,
\item[\rm{(3)}]
and is generically reduced.
\end{enumerate}
\end{thm}
\begin{proof}
We retain all the notations used in Lemma \ref{easy}; $\mathcal{H}$ is an extra component of $\mathcal{H}_{g+1,g,4}$ other than $\mathcal{H}^\mathcal{L}_{g+1,g,4}$,  $\mathcal{G}\subset\mathcal{G}^r_d$ is the component corresponding to $\mathcal{H}$,
$\h{W}\subset \h{W}^r_{g+1}$ is the component corresponding to $\h{G}$ where 
$r=\dim|\h{D}|$ for a general $(p,\h{D})\in\h{G}$ and so on.

\vni
{\textbf[$g=11$]} Note that $r=5$ since there does not exist a special $g^6_{g+1}=g^6_{12}$ on a non-hyperelliptic curve by Clifford's theorem; a hyperelliptic curve cannot be embedded in $\PP^r$, $r\ge 2$ as a curve of degree $g+1$ with a special hyperplane series.
We also note that a general element of $\h{W}^\vee$ is neither very ample nor birationally very ample; if so, the base-point-free part of a general $\mathcal{E}=g^3_8\in \mathcal{W}^\vee$ induces a birational morphism (or an embedding) into $\mathbb{P}^3$ and the genus of $C$ is at most $\pi(8,3)=9$ by the Castelnuovo genus bound. Therefore a general element of $\mathcal{W}^\vee$ is compounded.  We further note that 
$\mathcal{E}$ is base-point-free, otherwise $\mathcal{E}$ has at least degree two base locus in which case the Clifford's theorem applies. Therefore $\mathcal{E}$ induces degree two morphism $C\stackrel{\eta}{\rightarrow}E$ onto an elliptic curve $E\subset \mathbb{P}^3$ and $\h{E}=\eta^*(\h{F})$ for some non-special $\h{F}=g^3_4$ on the elliptic curve $E$, which is impossible by Lemma \ref{double}. Hence  there is no extra component of $\mathcal{H}_{g+1,g,4}$ other than  $\mathcal{H}^\mathcal{L}_{g+1,g,4}$, which is already irreducible of expected dimension.

\vni
One easily sees that $\mathcal{H}_{12,11,4}=\mathcal{H}^\mathcal{L}_{12,11,4}$ is non-empty. A non-singular model of a plane curve of degree $8$ with one ordinary $4$-fold points and $4$ nodes embeds into a Del Pezzo surface in $\mathbb{P}^4$ by the linear system of cubics through these $5$ points  as a smooth curve of degree $d=3\cdot 8-1\cdot  4-4\cdot 2=12$ and genus $g=21-6-4\cdot 1=11$; cf. \cite[Theorem 1.1]{rathman}.

\vskip 4pt
\noindent
{\textbf[$g=13$]} For $g=13$, we also have $r=\dim|\h{D}|=5$ for a general $(p,\h{D})\in\h{G}$ since $\pi (14,r)\le 11$ for $r\ge 6$. 

\vskip 4pt
\noindent
(a)
Assume that the morphism induced by $\h{E}=|K_C-\mathcal{D}|=g^3_{10}$ is compounded. If $\h{E}$ is base-point-fee, then  
$\h{E}$ induces a double covering $C \stackrel{\eta}{\rightarrow} E\subset\PP^3$ onto a curve $E$ of genus at most two (by the Castelnuovo genus bound applied to $E$) and $\h{E}=g^3_{10}=\eta^*(\mathcal{F})$ for some $\mathcal{F}=g^3_5$, which is  complete, non-special and base-point-free linear series on $E$.  However this is impossible by 
Lemma \ref{double}.
\vni
Suppose that a general element $\h{E}\in\h{W}^\vee$ has a non-empty base locus $\Delta$.  

\vni
If $\deg\Delta=1$, then 
the moving part of $\h{E}=g^3_{10}$ is a $g^3_9$ inducing a 3-sheeted map $C\stackrel{\zeta}{\rightarrow} D$ onto a rational normal curve $D\subset \mathbb{P}^3$, i.e. $C$ is trigonal.
Let $g^1_3$ be the unique  trigonal pencil (by Castelnuovo-Severi inequality) on $C$. 
Note that  $$\h{E}=|K_C-\mathcal{D}|=|3g^1_3|+\Delta$$
and take $\Gamma\in C_2$ such that $\Delta+\Gamma\in g^1_3$.
We then have 
$$|\h{E}+\Gamma|=|K_C-\h{D}+\Gamma|=|4g^1_3|$$
and it follows that $|\h{D}-\Gamma|=|K_C-4g^1_3|=g^4_{12}$ and therefore 
$\h{D}$ is not very ample, a contradiction.

\vni
In case $\deg\Delta=2$, $|\h{E}-\Delta|=g^3_8$ induces a double covering onto an elliptic curve, which is impossible by Lemma \ref{double}.
The case $\deg\Delta\ge 3$ never occurs by obvious reasons; Clifford's theorem etc..
Therefore we conclude that a general element of $\mathcal{W}^{\vee}$ cannot be compounded.

\vni
(b) 
Suppose that the moving part of a  general element of $\mathcal{W}^{\vee}\subset\mathcal{W}^{3}_{g-3}$ is very ample and take  a general $|K_C-\mathcal{D}|=g^3_{10}\in\h{W}^\vee$. Note that the the genus $g=13$ is larger than the second Castelnuovo genus bound $\pi_1(10,3)$;
$$\pi_1(10,3)=\frac{(g-4)(g-5)}{6}=12<g=13<\pi(10,3)=16$$ and hence by basic Castelnuvo theory the curve embedded by a very ample $g^3_{10}$ must lie on a quadric surface in $\mathbb{P}^3$; cf. \cite[Corollary 3.14, page 97]{H1}.  However 
there is no integer pair $(a,b)$  satisfying $a+b=10$ and $(a-1)(b-1)=13=g$.

\vni
(c)
Therefore the moving part of a  general element of $\mathcal{W}^{\vee}\subset\mathcal{W}^{3}_{g-3}$ is birationally very ample and let $b$ be the degree of the base locus $B$ of a general element of $\mathcal{W}^{\vee}$. 
 By Lemma \ref{easy}(2), we have
$b=0$ and 
\begin{equation}\label{equal2}4g-23\le \dim\mathcal{W}=\dim\mathcal{W}^{\vee}\le4g-22.
\end{equation}

\vni
By Lemma \ref{diagram}, we have a irreducible closed locus $\mathcal{Z}\subset\mathcal{W}^2_{g-5}$ such that $\dim\h{Z}=\dim\h{W}^\vee=\dim\h{W}$. Note that the base-point-free part of general element of $\mathcal{Z}$  is not very ample; $g=13$ is not a genus of a smooth plane curve.
If the base-point-free part of general element of $\mathcal{Z}$  is birationally very ample, then by Proposition \ref{wrdbd}(b) (applied to the locus $\h{Z}$),  we have
$$\dim\h{Z}=\dim\mathcal{W}\le 3(g-5)+g-1-8=4g-24, $$ which is not compatible with the above (\ref{equal2}).
 Therefore the only remaining possibility is that a general element of $\mathcal{Z}\subset\mathcal{W}^2_{g-5}$ is compounded and let $\Delta$ be the base locus of a general element $\h{F}=g^2_8\in\h{Z}$ with $\delta=\deg\Delta$. One of the following may occur; note that $\delta\le 2$ and $\delta\neq 1$ by Clifford's theorem etc.

\noindent
\begin{enumerate}
\item[\rm{(1)}] $C$ is a double covering of a curve $E$ of genus $2$ with a two sheeted map $C\stackrel{\zeta} \rightarrow E$ and $\mathcal{F}=\zeta^*({g^2_4})$; $\delta=0$.

\noindent
\item[\rm{(2)}] $C$ is a double covering of a smooth plane quartic $F$ with a two sheeted map $C\stackrel{\eta}\rightarrow F$ and $\mathcal{F}=\eta^*(|K_F|)=\eta^*(g^2_4)$; in the case $\h{F}$ is
a pull-back of the special $|K_F|$ on the quartic $F$ and $\delta=0$.

\noindent
\item[\rm{(3)}] $C$ is bi-elliptic with a bi-elliptic covering $C\stackrel{\phi} \rightarrow E$ and $\mathcal{F}=\phi^*(g^2_3)+\Delta$; $\delta =2$. 
\noindent
\item[\rm{(4)}] 
$C$ is trigonal and $\mathcal{F}=2g^1_3+\Delta$; $\delta=2$.
\noindent
\item[\rm{(5)}] $C$ a $4$-gonal curve with $\mathcal{F}=2g^1_4$; $\delta=0$. 
\end{enumerate}
\vskip 4pt
\noindent
We now choose a general $(\h{F},p)\in\h{Z}\subset\h{W}^2_{g-5}$ and set $$|\h{F}+\Gamma|\in\h{W}^\vee\subset\h{W}^3_{g-3}$$ where $\Gamma =t+s\in C_2$ is obtained from a singularity of the image curve of the morphism  $\xi^{-1}(p)=C\rightarrow\PP^3$ induced by 
a birationally very ample and base-point-free $g^3_{10}=|\h{F}+\Gamma|=\h{E}=|K_C-\h{D}|\in\h{W}^\vee$. 

\vni
(1) and (3): One may argue that the residual series $|\mathcal{D}|=|K_C-\mathcal{F}-\Gamma|$ is not very ample, virtually in the same manner as in the proof of Lemma \ref{double} as follows. If $C\stackrel{\zeta}\rightarrow E$ is a double covering onto a curve $E$ of genus $2$, then $$|\h{F}|=\zeta^*(g^2_4)=|\h{E}-t-s|=|K_C-\h{D}-t-s|$$ 
and hence 
$$|\h{D}|=|K_C-\zeta^*(g^2_4)-t-s|.$$ We take $r'+s'\in C_2$ which is the conjugate divisor of the divisor $t+s\in C_2$ with respect to the double covering $\zeta$.
Then we have 
$$|\h{D}-t'-s'|=|K_C-\zeta^*(g^2_4)-t-s-t'-s'|=|K_C-\zeta^*(g^4_6)|=g^4_{12}$$
hence $|\h{D}|$ is not very ample. The case (3) is almost identical to (1) which we omit.

\vni
For the case (2), a calculation similar to the above or applying Proposition \ref{wrdbd}(c) to the locus $\h{Z}$ fails to work. Instead we argue as follows. Recall that under our current circumstance, the double covering $\eta$ onto a smooth plane quartic is induced by a (general) element of $\h{Z}\subset\h{W}^2_{g-5}=\h{W}^2_8$, which is a subseries of a birationally very ample base-point-free $\h{E}\in\h{W}^\vee\subset\h{W}^3_{10}$. To be precise, consider the following diagram which is adopted to the current situation which appeared in the proof of Lemma \ref{diagram};
\[
\begin{array}{ccc}
\mathcal{W}^{2}_{8}\underset{\mathcal{M}}{\times}\mathcal{W}_2\supset\Sigma & \stackrel{q}\dashrightarrow &
\mathcal{W}^\vee\subset\mathcal{W}^{3}_{10}\subset\mathcal{W}^{2}_{10}\\
\\
\quad\quad\quad\quad\quad\quad\dashdownarrow\vcenter{\rlap{$\scriptstyle{{\pi}}\,$}} & 
\\ 
\\
\quad\quad\quad\mathcal{W}^{2}_{8}\supset\mathcal{Z} & \quad\stackrel{\zeta}{\dashrightarrow}\quad&\mathcal{X}_{2,3} \subset\mathcal{M}_g
\end{array}
\]
 \vni
 where $\mathcal{X}_{n,\gamma}$ denotes the locus in $\mathcal{M}_g$ corresponding to curves
which are n-fold coverings of smooth curves of genus $\gamma$. By de Franchis theorem (cf. \cite{Sommese}), the fiber in $\mathcal{Z}$ of the rational map $\zeta$ over a general class $[C]\in\zeta (\h{Z})\subset\h{X}_{2,3}$ is finite. As we saw in the proof of Lemmma \ref{diagram}, $\pi$ is a generically finite map and hence over $(\h{E}',p)\in\h{Z}$ we have  finitely many $\h{O}_C(R+S)$'s in the second factor of the locus $\Sigma$ so that $\h{E}'\otimes\h{O}_C(R+S)\in\h{W}^\vee$.
Therefore upon choosing a general $[C]\in\zeta(\h{Z})\subset\h{X}_{2,3}$, we arrive at an element $\mathcal{E}\in\h{W}^\vee$. The ambiguities in choosing an element among the finite fiber of $\zeta$ as well as the fiber of $\pi$ do not affect the following dimension count. In other words, we arrive at finitely many 
$\h{E}\in\h{W}^\vee$ and all the $\h{E}$'s obtained in this way lie on $W^3_{10}(C)$. 
By the well-known Riemann's moduli count \cite[Satz 1]{Lange} $$\dim\mathcal{X}_{n,\gamma}\le 2g+(2n-3)(1-\gamma)-2,$$ one has
\begin{eqnarray*}
4g-23\le\dim\h{W}^\vee=\dim\mathcal{W}\le\dim\mathcal{X}_{2,3} 
\le 2g-4,
\end{eqnarray*}
a contradiction. 

\vni
(4) Since a trigonal curve  $C$ of genus $g\ge 5$ has a unique trigonal pencil $g^1_3$, we have the following well defined rational map induced from the construction of the sublocus $\h{Z}\subset\h{W}^2_8$.

\begin{align*}
& \h{Z}\quad\stackrel{\psi}{\dashrightarrow}\quad\h{W}^1_{g,3}\underset{\mathcal{M}}{\times}\h{W}_2\quad\stackrel{\pi_2}{\dashrightarrow}\quad\h{W}^1_3\quad\stackrel{\kappa}{\dashrightarrow}\quad\h{M}^1_{g,3}\\
&\h{F} \mapsto (g^1_3, |\h{F}-2g^1_3|=\Delta)
\end{align*}
where $\pi_2$ is the second projection and $\kappa$ is a natural birational map.
We now claim that the composition $\tau :\h{Z} \dashrightarrow \h{M}^1_{g,3}$ of the above maps  is a generically finite map.

\vni
Take the residual series of $|\h{D}|\in\h{W}$;
$$|K_C-\h{D}|=\mathcal{E}=g^3_{10}=|\h{F}+\Gamma|=|2g^1_3+\Delta+\Gamma|=|2g^1_3+\Delta+R+S|,$$ so that
$$|\h{D}|=|K_C-2g^1_3-(R+S)-\Delta|.$$
Note that $\dim|2g^1_3+U+V|=\dim|2g^1_3|=2$ for any $U+V\in C_2$ and hence $$\dim|K_C-2g^1_3-(R+S)|=\dim|K_C-2g^1_3|-2=6.$$ 
We also note that a very ample $\h{D}$ is a subseries of $|K_C-2g^1_3-(R+S)|$ and hence the series $|K_C-2g^1_3-(R+S)|$ is birationally very ample, which is also base-point-free.
Therefore it follows that there exists only finitely many choice of $\Delta$'s subject to the condition 
$$5=\dim|\h{D}|=\dim|K_C-2g^1_3-(R+S)-\Delta|=\dim|K_C-2g^1_3-(R+S)|-1.$$
and  hence $\tau$ is generically finite. 
By this  claim follows that 
$$4g-23\le\dim\h{Z}\le\dim\h{M}^1_{g,3}=2g+1,$$
which is a contradiction. 

\vni
(5) In this case we have $\h{F}=2g^1_4$ for a general $\h{F}\in\h{Z}$, which is complete. Note that on a  $4$-gonal curve $C$ of genus $g\ge10$ which is not bi-elliptic, there exists a unique $g^1_4$ and hence  we have a natural generically injective rational map 
$$\h{Z}\stackrel{\pi}{\dashrightarrow} \h{M}^1_{g,4}.$$
and it follows that  
$$4g-23\le\dim\h{Z}\le \dim\h{M}^1_{g,4}=2g+3$$ and hence the rational map $\pi$ is dominant. Recall that on a general $4$-gonal curve $C$, 
$|K_C-2g^1_4|$ is very ample; cf. \cite{B}. Hence it follows that 
\begin{eqnarray*}\dim|\h{D}|&=&\dim|K_C-\h{E}|=\dim|K_C-\mathcal{F}-\Gamma|\\&=&\dim|K_C-2g^1_4-R-S|=\dim|K_C-2g^1_4|-2=6-2=4
\end{eqnarray*} which is a contradiction.

\vni
{\textbf[$g=14$]}
As in the cases of lower $g$, we have $r=5$, otherwise the Castelnuovo bound for a very ample $g^6_{15}=g^6_{g+1}$ is $\pi(15,6)=13$ which is less than $g=14$.
\vni
(a) If a general element of $\mathcal{W}^{\vee}$ is compounded, then by Lemma \ref{easy} (3), $g\le 3r-2=13$, which is impossible.

\vni
(b) If a general element of $\h{W}^\vee\subset\h{W}^3_{11}$ is very ample, then $\dim\h{W}^\vee=4g-23$ by Lemma \ref{easy} (1). Indeed $4g-23$ is the maximal possible dimension of any $\h{V}\subset\h{W}^3_{11}$ consisting of very ample linear series by Proposition \ref{wrdbd} (a). 
Therefore it follows that over the locus $\h{W}^\vee\subset\h{W}^3_{11}$ generically consisting of very ample complete linear series, there is a component $\tilde{\h{H}}$ of the Hilbert scheme $\h{H}_{11,14,3}=\h{H}_{g-3,g,3}$ such that 
$$\dim\tilde{\h{H}}=\dim\h{W}^\vee +\dim\PP GL(4)=4g-23+15=48,$$
which is impossible 
by 
Theorem \ref{residual}.

\vni
(c) Therefore a general element of $\h{W}^\vee$ is birationally very ample which base-point-free by Lemma \ref{easy} (2). By Lemma \ref{diagram}, we take the locus $\h{Z}\subset\h{W}^2_{g-5}$ such that
$$4g-23\le\dim\h{Z}=\dim\h{W}^\vee=\dim\h{W}\le 4g-22.$$

\vni
(i) Base-point-free part of a general element of $\h{Z}$ is not very ample since $g=14$ is not a genus of a smooth plane curve.

\vni
(ii) Base-point-free part of a general element of $\h{Z}$ birationally very ample by Proposition \ref{wrdbd} (2);  if so  $4g-23\le \dim\h{Z}\le 3(g-5)+g-1-4\cdot 2=4g-24$, a contradiction.

\vni
(iii) Therefore the base-point-free part of a general element of $\h{Z}$ is compounded. We list up all the possibilities according to the degree $\delta$ of the base locus $\Delta$ of a general $\h{F}\in\h{Z}$.

\begin{enumerate}
\item[\rm{(1)}] $\delta=0$; A general $(\h{F}, p)\in\h{Z}$ induces a triple covering $C=\xi^{-1}(p)\stackrel{\tau}{\rightarrow}E$ onto an elliptic curve. 

\item[\rm{(2)}] $\delta=1$; A general $\h{F}\in\h{Z}$ induces a double covering $C\stackrel{\tau}{\rightarrow}E$
onto a curve of genus $2$.

\item[\rm{(3)}] $\delta=1$; A general $\h{F}\in\h{Z}$ induces a double covering $C\stackrel{\tau}{\rightarrow}E$
onto a non-hyperelliptic curve of genus $3$.

\item[\rm{(4)}] $\delta=1$; A general $\h{F}\in\h{Z}$ induces a $4$-sheeted covering $C\stackrel{\tau}{\rightarrow}E$
onto a rational curve.

\item[\rm{(5)}] $\delta=3$; A general $\h{F}\in\h{Z}$ induces a double covering $C\stackrel{\tau}{\rightarrow}E$
onto an elliptic curve.

\item[\rm{(6)}] $\delta=3$; A general $\h{F}\in\h{Z}$ induces a triple covering $C\stackrel{\tau}{\rightarrow}E$
onto a rational curve.
\end{enumerate}

\vni
Instead of  carrying out a precise dimension estimate in all the cases above as we did in the case $g=13$, we will make straight forward dimension estimate this time. Our estimate is rather crude but this is sufficient for our purpose.
By assuming that a general element of $\h{Z}$ is compounded, we have a sequence of rational maps defined as follows.

\begin{align*}
&{\hskip -8pt}{\h{W}^2_{g-5}}&&{\hskip -12pt\h{W}^2_{g-5-\delta}\underset{\mathcal{M}}{\times} \h{W}_\delta}\\
&\hskip -2pt\downinclusion&&\hskip 8pt\downinclusion\\
&\h{Z} \quad\quad\quad\stackrel{\psi}{\dashrightarrow }&&\quad\Lambda\quad\quad\stackrel{\varphi}{\dashrightarrow} \quad\h{K}=\h{X}_{n,\gamma}\quad\hookrightarrow\quad\h{M}_g\\
&& \vequal\hskip -17pt
\\&&\{(\h{F}',\Delta)|\h{F}'+\Delta=\h{F}\in\h{Z}\}\hskip -70pt
\\&\downelement&&\hskip 8pt\downelement
\\&\h{F} \hskip 30pt\longmapsto &&(\h{F}',\Delta)
\end{align*}

\vni
$\psi$ is the map sending $\h{F}$ to the pair $(\h{F}', \Delta))$ where $\h{F}'$ is the moving part of $\h{F}$ and $\Delta$ is the base locus.
The second dotted arrow $\varphi$ is the map assiging the base-point-free part $\h{F}'$ to its isomorphism class in an appropriate $\h{X}_{n,\gamma}$ determined by the compounded $\h{F}'$. It is clear that the first arrow $\psi$ is well defined and generically injective. We stress that $\varphi$ is also well defined by the Castelnuovo-Severi inequality so that there is only one choice of multiple covering upon choosing a general $\h{F}\in\h{Z}$. However the fiber of $\varphi$ over  a point in the image may have dimension more than the degrees of freedom choosing the  base locus.

\vni For example, 
in the case (1), $\Delta=\emptyset$ and there is no second factor
and $\h{K}=\h{X}_{3,1}\subset\h{M}_g$. Given a triple cover of an elliptic curve $C\stackrel{\tau}{\rightarrow}E$ represented by a point  $p\in\h{X}_{3,1}$, the fiber of $\varphi$ over $p$  is inside $\tau^*(W^2_3(E))=\tau^*(J(E))$ which is one dimensional. 
Hence it follows that $$4g-23\le\dim\h{Z}\le\dim\tau^*(J(E))+\dim\h{X}_{3,1}\le 1+ 2g-2$$ leading to an absurdity. For the remaining possible cases, we may come up with similar numerical absurdities as follows.
\vni
\begin{enumerate}

\item[\rm{(2)}] $\delta=1$, a general $\h{F}\in\h{Z}$ induces a double covering $C\stackrel{\tau}{\rightarrow}E$ onto a curve $E$ of genus $2$ and $\h{F}=\tau^*(g^2_4)+\Delta$.
$$4g-23\le\dim\h{Z}\le\dim\tau^*(J(E))+\dim\h{X}_{2,2}+\delta\le 2+ 2g-3+ 1=2g$$ 

\vni
\item[\rm{(3)}] 
$\delta=1$; a general $\h{F}\in\h{Z}$ induces a double covering $C\stackrel{\tau}{\rightarrow}E$ onto a non-hyperelliptic curve of genus $3$ and and $\h{F}=\tau^*(|K_E|)$.
$$4g-23\le\dim\h{Z}\le\dim\tau^*(|K_E|))+\dim\h{X}_{2,3}+\delta\le  2g-4+\delta=2g-3$$ 

\vni
\item[\rm{(4)}] $\delta=1$; a general $\h{F}\in\h{Z}$ induces a $4$-sheeted covering $C\stackrel{\tau}{\rightarrow}E$
onto a rational curve.
$$4g-23\le\dim\h{Z}\le\dim\h{M}^1_{g,4}+\delta \le 2g+3+1=2g+4$$

\vni
\item[\rm{(5)}] $\delta=3$; a general $\h{F}\in\h{Z}$ induces a double covering $C\stackrel{\tau}{\rightarrow}E$
onto an elliptic curve.
$$4g-23\le\dim\h{Z}\le\dim\tau^*(J(E))+\dim\h{X}_{2,1}+\delta\le 2g+2$$  

\vni
\item[\rm{(6)}] $\delta=3$; A general $\h{F}\in\h{Z}$ induces a triple covering $C\stackrel{\tau}{\rightarrow}E$
onto a rational curve.
$$4g-23\le\dim\h{Z}\le\dim\h{M}^1_{g,3}+\delta\le 2g+1+3=2g+4$$
\end{enumerate}
We could have used a variation of Lemma \ref{double} in the cases (2) and (5) to deduce that $\h{D}$ from which $\h{F}$ is induced is not very ample,  as we did in the case for $g=13$. Also note that the above estimate is rather rough, i.e. the possible choice of the base locus is usually finite under several other conditions of ours. 

\vni
For all the three Hilbert schemes we treated, we showed $\h{H}_{g+1,g,4}=\h{H}^\h{L}_{g+1,g,4}$ and hence $\h{H}_{g+1,g,4}$ is irreducible since $\h{H}^\h{L}_{g+1,g,4}$ is. On the other hand, in \cite[Theorem 2.1]{KK3}, one proves the irreducibility of $\h{H}^\h{L}_{g+1,g,4}$ by showing that 
the locus $\h{G}_\h{L}$ corresponding to a component of $\h{H}^\h{L}_{g+1,g,4}$ is birational to the locus $\h{G}'\subset\mathcal{G}^2_{g-3}$ corresponding to the Severi variety  $\Sigma_{g-3,g}$, which is irreducible and has the expected dimension. Therefore $\h{H}_{g+1,g,4}=\h{H}^\h{L}_{g+1,g,4}$ has the expected dimension.The generically reducedness of $\h{H}_{g+1,g,4}=\mathcal{H}^\mathcal{L}_{g+1,g,4}$ follows from the fact that the dimension of the singular locus of $\mathcal{G}'\subset \mathcal{G}^{2}_{g-3}$ does not exeed $g-8<\lambda(g+1,g,4)=\lambda(g-3,g,2)=4g-18$; cf. \cite[Proposition (2.9)]{AC2}. Thus $\mathcal{G}'$ is generically reduced and since $\mathcal{G}_\mathcal{L}$ is birational to $\mathcal{G}'$ it follows that $\h{H}_{g+1,g,4}=\mathcal{H}_{g+1,g,4}^\mathcal{L}$ is also generically reduced. 
\end{proof}

\begin{rmk} (1) It should be remarked that for $g=15$, one can show the irreducibility of $\h{H}_{g+1,g,4}$ by a similar method we used for the case $g=14$. 
As an intermediate step (such as Theorem 2.6), one may show that  the Hilbert scheme $\h{H}_{12,15,3}=\h{H}_{g-3,g,3}$ is of the minimal possible dimension $4(g-3)$ (and is irreducible in this case) by using the fact that $\h{H}_{8,5,3}$ is irreducible of dimension $4\cdot 8$ whose general element is directly linked to the one in $\h{H}_{12,15,3}$ via complete intersection of a quartic and a quintic; recall that the irreducibility of $\h{H}_{8,5,3}$ is known by \cite{E1} or \cite{KK}. It is worthwhile to  remark that there is no component of  $\h{H}_{12,15,3}$ whose general element corresponds to a curve $C$ on a quadric or a cubic surface. For example, one can eliminate the possibility for a general element of (a component of) $\h{H}_{12,15,3}$ lying on a smooth cubic as follows. Note that smooth space curves of degree $d$ and genus $g$ on a smooth cubic surface form a finite union of locally closed irreducible family in $\h{H}_{d,g,3}$ of dimension $d+g+18$ if $d\ge 10$ by \cite[Proposition B.1]{Gruson}.  Since $d+g+18< 4d$ for $(d,g)=(12,15)$, it follows that this family does not constitute a component. In fact, we could have used \cite[Proposition B.1]{Gruson} directly in the course of the proof of Theorem \ref{residual} instead of going through thorough computation. We leave the other details for interested readers. 
\vni
(2) Note that $\rho (g+1,g,4)=0$ for $g=15$ and one may expect that the same tactics using linkage theory as above or as the case $g=14$ may work for $g\ge 16$ in general. However, when the genus $g$ and the degree $d$ of the family of curves in question is large, the curve often needs to lie on surfaces of rather high degree. Usually such curves are not necessarily linked to another curve which is easier to describe or we know much of.
\end{rmk}

\bibliographystyle{spmpsci} 

\end{document}